\documentclass[12pt,a4paper]{article}%
\usepackage{graphicx}
\usepackage{amsmath}%
\usepackage{amsfonts}%
\usepackage{amssymb}
\newtheorem{theorem}{Theorem}

\newtheorem{claim}[theorem]{Claim}

\newtheorem{lemma}[theorem]{Lemma}

\newenvironment{proof}[1][Proof]{\textbf{#1.} }{\ \rule{0.5em}{0.5em}}

\begin{document}

\title{Pursuit-Evasion Games with Incomplete Information in Discrete Time}
\author{Ori Gurel-Gurevich \thanks{Weizmann Institute of Science,
Rehovot, 76100, Israel. e-mail: ori.gurel-gurevich@weizmann.ac.il}}
\maketitle

\begin{abstract}
Pursuit-Evasion Games (in discrete time) are stochastic games with
nonnegative daily payoffs, with the final payoff being the
cumulative sum of payoffs during the game. We show that such games
admit a value even in the presence of incomplete information and
that this value is uniform, i.e. there are $\epsilon$-optimal
strategies for both players that are $\epsilon$-optimal in any long
enough prefix of the game. We give an example to demonstrate that
nonnegativity is essential and expand the results to Leavable Games.
\end{abstract}

\bigskip


\section{Introduction}

Games of Pursuit and Evasion are two-player zero-sum games
involving a Pursuer (P) and an Evader (E). P's goal is to capture
E, and the game consist of the space of possible locations and the
allowed motions for P and E. These games are usually encountered
within the domain of differential games, i.e., the location space
and the allowed motions have the cardinality of the continuum and
they tend to be of differentiable or at least continuous nature.

The subject of Differential Games in general, and Pursuit-Evasion
Games in particular, was pioneered in the 50s by Isaacs (1965).
These games evolved from the need to solve military problems such
as airfights, as opposed to classical game theory which was
oriented toward solving economical problems. The basic approach
was akin to differential equations techniques and optimal control,
rather than standard game theoretic tools. The underlying
assumption was that of complete information, and optimal
\emph{pure} strategies were searched for. Conditions were given,
under which a pure strategies saddle point exists (see, for
example, Varaiya and Lin (1969)). Usually the solution was given
together with a value function, which assigned each state of the
game its value. Complete information was an essential requirement
in this case. For a thorough introduction to Pursuit-Evasion and
Differential Games see Basar and Olsder (1999).

A complete-information continuous-time game ``intuitively'' shares
some relevant features with \emph{perfect-information}
discrete-time games. The latter are games with complete knowledge
of past actions and without simultaneous actions. Indeed, if one
player decides to randomly choose between two pure strategies
which differ from time $t_0$ and on, his opponent will discover
this ``immediately'' after $t_0$, thus enabling himself to respond
optimally almost instantly. Assuming the payoff is continuous, the
small amount of time needed to discover the strategy chosen by the
opponent should affect the payoff negligibly. A well-known result
of Martin (1975, 1985) implies that every perfect-information
discrete-time game has $\epsilon$-optimal pure strategies
(assuming a Borel payoff function) and so should, in a sense,
continuous time games.

Another reason to restrict oneself to pure strategies is that
unlike discrete-time games, there is no good formal framework for
continuous-time games. By framework we mean a way to properly
define the space of pure strategies and the measurable
$\sigma$-algebra on them. There are some approaches but none is as
general or complete as for discrete-time games. This kind of
framework is essential when dealing with a general incomplete
information setting.

This paper will therefore deal with discrete-time Pursuit-Evasion
Games. We hope that our result will be applied in the future to
discrete approximations of continuous-time games. Pursuit-Evasion
Games in discrete time are formalized and discussed in Kumar and
Shiau (1981).

Pursuit-Evasion Games are generally divided into two categories:
\emph{Games of Kind} and \emph{Games of Degree}. Games of Kind
deal with the question of \emph{capturability}: whether a capture
can be achieved by the Pursuer or not. In a complete-information
setting this is a yes-or-no question, completely decided by the
rules of the game and the starting positions. With incomplete
information incorporated, we simply assign a payoff of 1 for the
event of capture and payoff 0 otherwise. Games of Degree have the
Pursuer try to minimize a certain payoff function such as the time
needed for capture. The question of capturability is encountered
here only indirectly: if the Evader have a chance of escaping
capture indefinitely, the expected time of capture is infinity.
The payoff, in general, can be any function, such as the minimal
distance between the Evader and some target set.

What unites the two categories is that the payoff function in both
is positive and cumulative. The maximizing player, be it the
Pursuer or the Evader, gains his payoff and never loses anything.
This is in contrast with other classes of infinitely repeated
games, such as undiscounted stochastic games, where the payoff is
the limit of the averages of daily payoffs.

Discrete-time stochastic games were introduced by Shapley (1953)
who proved the existence of the discounted value in two-player
zero-sum games with finite state and action sets. Recursive games
were introduced by Everett (1957). These are stochastic games, in
which the payoff is 0 except for absorbing states, when the game
terminates. Thus, absorbing states are as happens in
Pursuit-Evasion Games, where the payoff is obtained only when the
game terminates. The game is said to have a uniform value if
$\epsilon$-optimal strategies exist that are also
$\epsilon$-optimal in any long enough prefix of the game. Everett
proved the existence of the uniform value for two-player, zero-sum
recursive games.

We shall now formally define Pursuit-Evasion Games to be
two-player zero-sum games with cumulative and positive payoffs. To
avoid confusion, the players will be called the Maximizer and the
Minimizer, and their respective goals should be obvious.

Our main result is the existence of uniform value for
Pursuit-Evasion Games with incomplete-information and finite
action and signal sets, followed by a generalization for arbitrary
signal sets. In section 4 we present a different class of games to
which our proof also applies. In section 5 we show that the
positiveness requirement is indispensable by giving an appropriate
counterexample.

\section{Definitions and the main Theorem}

A \emph{cumulative game with complete information} is given by:
\begin{itemize}
\item Two finite sets $A^1$ and $A^2$ of actions.

Define $H_n=(A^1\times A^2)^n$ to be the set of all histories of
length $n$, and $H=\cup_{n=0}^\infty H_n$ to be the set of all
finite histories.

\item A daily payoff function $f:H\rightarrow\mathbb{R}$.

\end{itemize}

Let $\widetilde{H}=(A^1\times A^2)^{\aleph_0}$ be the set of all
infinite histories. The daily payoff function induces a payoff
function $\rho:\widetilde{H}\rightarrow\mathbb{R}$ by
$\rho(h)=\sum_{n=0}^\infty f(h_n)$, where $h_n$ is the length $n$
prefix of $h$. In the sequel we will only study the case in which
$f$ is nonnegative, so that $\rho$ is well defined (though it may
be infinite).

The game is played in stages as follows. The initial history is
$h_0=\emptyset$. At each stage $n\geq 0$ both players choose
simultaneously and independently actions $a\in A$ and $b\in B$,
and each player is informed of the other's choice. The new game
history is $h_{n+1}=h_n\frown<a,b>$, i.e., the concatenation of
$<a,b>$ to the current history. The infinite history of the game,
$h$, is the concatenation of all pairs of actions chosen
throughout the game. The payoff is $\rho(h)$, the goal of the
Maximizer is to maximize the expectation of $\rho(h)$, and that of
the Minimizer is to minimize it.

If all the values of $f$ are nonnegative, we call the game
\emph{nonnegative}. A \emph{complete information Pursuit-Evasion
Game} is a nonnegative cumulative game.

As cumulative games are a proper superset of recursive games (see
Everett (1957)), Pursuit-Evasion Games are a proper superset of
nonnegative recursive games.

As is standard in game theory, the term ``complete information''
is used to denote a game with complete knowledge of the history of
the game, and not the lack of simultaneous actions (which is
termed ``perfect information'').

A \emph{cumulative game with incomplete information} is given by:
\begin{itemize}
\item Two finite sets $A^1$ and $A^2$ of actions.

Define $H_n$ and $H$ as before.

\item A daily payoff function $f:H\rightarrow\mathbb{R}$.

\item Two measure spaces $S^1$ and $S^2$ of signals.

\item $\forall h\in H$ two probability distributions $p_h^1 \in
\Delta(S^1)$ and $p_h^2 \in \Delta(S^2)$.

\end{itemize}

Define $\widetilde{H}$ and $\rho$ as before. In particular, the
signals are not a parameter of the payoff function.

An incomplete-information cumulative game is played like a
complete information cumulative game, except that the players are
not informed of each other's actions. Instead, a signal pair
$<s^1,s^2>\in S^1\times S^2$ is randomly chosen with distribution
$p_h^1\times p_h^2$, $h$ being the current history of the game,
with player $i$ observing $s^i$. An \emph{incomplete-information
Pursuit-Evasion Game} is an incomplete-information nonnegative
cumulative game.

Define $H^i_n$ to be $(A^i\times S^i)^n$. This is the set of
private histories of length $n$ of player $i$. Similarly, define
$H^i=\cup_{n=0}^\infty H^i_n$, the set of all private finite
histories, and $\widetilde{H}^i=(A^i\times S^i)^{\aleph_0}$ the
set of all private infinite histories.

In a complete-information cumulative game a behavioral
\emph{strategy} for player $i$ is a function
$\sigma^i:H\rightarrow \Delta(A^i)$. In an incomplete-information
cumulative game a (behavioral) \emph{strategy} for player $i$ is a
function $\sigma^i:H^i\rightarrow\Delta(A^i)$. Recall that by
Kuhn's Theorem (Kuhn (1953)) the set of all behavioral strategies
coincides with the set of all mixed strategies, which are
probability distributions over pure strategies.

Denote the space of all behavioral strategies for player $i$ by
$\Omega^i$. A \emph{profile} is a pair of strategies, one for each
player. A profile $<\sigma^1,\sigma^2>$, together with
$\{p_h^i\}$, induces, in the obvious manner, a probability measure
$\mu_{\sigma^1,\sigma^2}$ over $\widetilde{H}$ equipped with the
product $\sigma$-algebra.

The \emph{value} of a strategy $\sigma^1$ for the Maximizer is
$val(\sigma^1)=\inf_{\sigma^2 \in \Omega^2}
E_{\mu_{\sigma^1,\sigma^2}}(\rho(h))$.

The \emph{value} of a strategy $\sigma^2$ for the Minimizer is
$val(\sigma^2)=\sup_{\sigma^1 \in \Omega^1}
E_{\mu_{\sigma^1,\sigma^2}}(\rho(h))$.

When several games are discussed we will explicitly denote the
value in game $G$ by $val_G$.

The \emph{lower value} of the game is
$\underline{val}(G)=\sup_{\sigma^1 \in \Omega^1} val(\sigma^1)$.

The \emph{upper value} of the game is
$\overline{val}(G)=\inf_{\sigma^2 \in \Omega^2} val(\sigma^2)$.

If $\underline{val}(G)=\overline{val}(G)$, the common value is the
\emph{value} of the game
$val(G)=\underline{val}(G)=\overline{val}(G)$. Observe that
$\underline{val}(G)$ and $\overline{val}(G)$ always exist, and
that $\underline{val}(G)\leq\overline{val}(G)$ always holds.

A strategy $\sigma^i$ of player $i$ is \emph{$\epsilon$-optimal}
if $|val(\sigma^i)-val(G)|<\epsilon$. A strategy is \emph{optimal}
if it is 0-optimal.

A cumulative game is \emph{bounded} if its payoff function $\rho$
is bounded, i.e. $\exists B\in\mathbb{R}\forall h \in
\widetilde{H} \ \ \ -B<\rho(h)<B$.

Let $G=<A^1,A^2,f>$ be a cumulative game. Define $f_n$ to be equal
to $f$ for all histories of length up to $n$ and zero for all
other histories. Define $G_n=<A^1,A^2,f_n>$. Thus, $G_n$ is the
restriction of $G$ to the first $n$ stages. Let $\rho_n$ be the
payoff function induced by $f_n$.

A game $G$ is said to have a \emph{uniform value} if it has a
value and for each $\epsilon>0$ there exist $N$ and two strategies
$\sigma^1,\sigma^2$ for the two players that are
$\epsilon$-optimal for every game $G_n$ with $n>N$.

The first main result is:

\begin{theorem}
Every bounded Pursuit-Evasion Game with incomplete-information and
finite signal sets has a uniform value. Furthermore, an optimal
strategy exists for the Minimizer.
\end{theorem}

\begin{proof}
Let $G$ be a bounded Pursuit-Evasion Game with
incomplete-information . Let $G_n$ be defined as above. Since
$A^1,A^2,S^1,S^2$ are all finite, there are only a finite number
of private histories of length up to $n$. $G_n$ is equivalent to a
finite-stage finite-action game, and therefore it has a value
$v_n$. From the definition of $G_n$ and since $f$ is nonnegative
\[
\forall h\in\widetilde{H} \ \
\rho_n(h)\leq\rho_{n+1}(h)\leq\rho(h)
\]
which implies that for all $\sigma^1\in\Omega^1$
\begin{equation}
val_{G_n}(\sigma^1)\leq val_{G_{n+1}}(\sigma^1) \leq
val_G(\sigma^1)
\label{eq0}
\end{equation}
so that
\[
\underline{val}(G_n)\leq \underline{val}(G_{n+1}) \leq
\underline{val}(G)  .
\]

Therefore, $v_n$ is a nondecreasing bounded sequence and
$\underline{val}(G)$ is at least $v=\lim_{n\rightarrow\infty}v_n$.

On the other hand, define $K_n=\{\sigma^2\in\Omega^2 \mid
val_{G_n}(\sigma^2)\leq v\}$. Since $val(G_n)=v_n\leq v$, $K_n$
cannot be empty.

$K_n$ is a compact set, since the function $val_{G_n}(\sigma^2)$
is continuous over $\Omega^2$, which is compact, and $K_n$ is the
preimage of the closed set $(-\infty,v]$.

For all $\sigma^2\in\Omega^2$ $val_{G_n}(\sigma^2)\leq
val_{G_{n+1}}(\sigma^2)$, so that  $K_n\supseteq K_{n+1}$. Since
the sets $K_n$ are compact, their intersection is nonempty.

Let $\sigma^2$ be a strategy for the Minimizer in
$\cap_{n=0}^\infty K_n$. Let $\sigma^1$ be any strategy for the
Maximizer. From $\rho(h)=\lim_{n\rightarrow\infty}\rho_n(h)$ and
since $\rho$ is bounded, we get by the monotone convergence
Theorem
\[
E_{\mu_{\sigma^1,\sigma^2}}(\rho(h))
=\lim_{n\rightarrow\infty}E_{\mu_{\sigma^1,\sigma^2}}(\rho_n(h)) .
\]
Since $\sigma^2$ belongs to $K_n$,
$E_{\mu_{\sigma^1,\sigma^2}}(\rho_n(h))\leq v$ and therefore
$E_{\mu_{\sigma^1,\sigma^2}}(\rho(h))\leq v$. Since $\sigma^1$ is
arbitrary $val(\sigma^2) \leq v$, so that $\overline{val}(G) \leq
v$. Consequentially, $v$ is the value of $G$.

Notice that any $\sigma^2 \in \cap_{n=0}^\infty K_n$ has
$val_G(\sigma^2)=v$ and is therefore an optimal strategy for the
Minimizer.

Given $\epsilon>0$ choose $N$ such that $v_{N}>v-\epsilon$. Let
$\sigma^1$ be an optimal strategy for the Maximizer in $G_N$, and
let $\sigma^2 \in \cap_{n=0}^\infty K_n$. By (\ref{eq0})
\[
\forall n>N \ \ v_n-\epsilon\leq v-\epsilon <
v_N=val_{G_N}(\sigma^1)\leq val_{G_n}(\sigma^1)
\]
so that $\sigma^1$ is $\epsilon$-optimal in $G_n$. As $\sigma^2\in
K_n$ one has $val_{G_n}(\sigma^2)\leq v < v_n+\epsilon$ so that
$\sigma^2$ is $\epsilon$-optimal in $G_n$.

These strategies are $\epsilon$-optimal in all games $G_n$ for
$n>N$. Thus, the value is uniform.
\end{proof}

Remark: Most of the assumption on the game $G$ are irrelevant for
the proof of the theorem and were given only for the simplicity of
description.

\begin{enumerate}
\item The action sets $A^i$ and the signal sets $S^i$ may depend
respectively on the private histories $H^i_n$.

\item The signals $<s^1,s^2>$ may be correlated, i.e. chosen from
a common distribution $p_h \in \Delta(S^1 \times S^2)$.

\item The game can be made stochastic simply by adding a third
player, Nature, with a known behavioral strategy. The action set
for Nature can be countable, since it could always be approximated
by large enough finite sets. The action sets for the Maximizer can
be infinite as long as the signals set $S^2$ is still finite (so
the number of pure strategies for the Minimizer in $G_n$ is still
finite).

\item Since the bound on payoffs was only used to bound the values
of $G_n$, one can drop the boundedness assumption, as long as the
sequence $\{v_n\}$ is bounded. If they are unbounded then $G$ has
infinite uniform value in the sense that the Maximizer can achieve
as high a payoff as he desires.
\end{enumerate}

\section{Arbitrary signal sets}

Obviously, the result still hold if we replace the signal set $S$ by
a sequence of signal sets $S_n$, all of which are finite, such that
the signals for histories of length $n$ belong to $S_n$. The signal
sets, like the action sets can change according to past actions, but
since there are only finitely many possible histories of length $n$,
this is purely semantical.

What about signals chosen from an infinite set? If the set $S$ is
countable than we can approximate it with finite sets $S_n$, chosen
such that for any history $h$ of length $n$ the chance we get a
signal outside $S_n$ is negligible. We won't go into details because
the next argument applies for both the countable and the uncountable
cases.

A cumulative game $G$ is \emph{$\epsilon$-approximated} by a game
$G'$ if $G'$ has the same strategy spaces as $G$ and for any pair
of strategies $\sigma,\tau$
\[
|\rho_G(\sigma,\tau)-\rho_{G'}(\sigma,\tau)|<\epsilon  .
\]
\begin{lemma}
If $G$ is a bounded Pursuit-Evasion Game with incomplete
information then $G$ can be $\epsilon$-approximated by a Pursuit
Evasion Game with incomplete information with the same action sets
and payoffs which can be simulated using a sequence of finite
signal sets.
\end{lemma}

\begin{proof}
Let $G$ be such a game. Assume, w.l.o.g., that the payoff function
$\rho$ is bounded by 1. Fix a positive $\epsilon$. Let
$\epsilon_n=\epsilon/{2^n}$. Define $p_n^i=\sum_{h\in
H_n}p_h^i/|H_n|$, the mean distribution of the signals at stage
$n$. Every distribution $p_h^i$ of time $n$ is absolutely
continuous with respect to $p_n^i$. By Radon-Nykodim theorem, a
density function $f_h^i$ exists such that $p_h^i(E)=\int_E f_h^i
dp_n^i$. Clearly, $f_h^i$ is essentially bounded by $|H_n|$.

Let $S'^i_n$ be $\{0,\epsilon_n,2\epsilon_n,3\epsilon_n,...,
\lfloor |H_n|/\epsilon_n \rfloor \epsilon_n\}^{|H_n|}$. For $h\in
H_n$ define $f'^i_h$ to be $f_h^i$ rounded down to the nearest
multiple of $\epsilon_n$. Define $F'^i_n:S^i\rightarrow S'^i_n$ by
$F'^i_n(s)=\{f'^i_h(s)\}_{h\in H_n}$. Let $G'$ be the same game as
$G$ except that the players observe the signals $F'^i_n(s^i)\in
S'^i_n$ where $s^i$ is the original signal with density $f^i_h$.

Given a signal $s'^i$ in $S'^i_n$ one can project it back onto
$S^i$ by choosing from a uniform distribution (with respect to the
measure $p_n^i$) over the set $E(s'^i)={F'^i_n}^{-1}(s'^i)$. Let
$G''$ be the game $G$ except that the signals are chosen with the
distribution just described. Denote their density function by
$f''^i_h$. This game can be simulated using only the signals in
$G'$ and vice versa so they are equivalent.

$G$ and $G''$ have exactly the same strategy spaces. The only
difference is a different distribution of the signals. But the way
the signals in $G''$ were constructed it is obvious that the
density function $f''^i_h$ do not differ from $f^i_h$ by more than
$\epsilon_n$ for any history $h$ of length $n$. Given a profile
$<\sigma^1,\sigma^2>$ denote the generated distributions on
$\widetilde{H}$ in $G$ and $G''$ by $\mu$ and $\mu''$. The payoffs
are $\rho_G(\sigma^1,\sigma^2)=\int \rho d\mu$ and
$\rho_{G''}(\sigma^1,\sigma^2)=\int \rho d\mu''$ . But the
distance, in total variation metric, between $\mu$ and $\mu''$
cannot be more than the sum of distances between the distributions
of signals at each stage, which is no more than $\sum_{i=1}^\infty
\epsilon_i=\epsilon$. By definition of total variation metric, the
difference between $\int \rho d\mu$ and $\int \rho d\mu''$ cannot
be more than $\epsilon$.
\end{proof}

\begin{theorem}
If $G$ is as in lemma and have bounded nonnegative payoffs, it has
a uniform value.
\end{theorem}
\begin{proof}
Let $G$ be such a game, and for any $\epsilon$ let $G_\epsilon$ be
an $\epsilon$-approximation of $G$ produced by the lemma.
$G_\epsilon$ is equivalent to a game with finite signal sets and
therefore has a value according to Theorem 1, denoted
$v_\epsilon$. It is immediate from the definition of
$\epsilon$-approximation that $\underline{v}$, the lower value of
$G$ cannot be less than $v_\epsilon - \epsilon$, and likewise
$\overline{v}$ is no more than $v_\epsilon + \epsilon$.
$\overline{v}-\underline{v}$ is therefore less than $2\epsilon$.
But $\epsilon$ was chosen arbitrarily, so that
$\overline{v}=\underline{v}$.

Given $\epsilon>0$ let $\sigma^1$ and $\sigma^2$ be
$\epsilon/2$-optimal strategies in $G_{\epsilon/2}$ that are also
$\epsilon/2$-optimal in any prefix of $G_{\epsilon/2}$ longer than
$N$. Clearly, these strategies are $\epsilon$-optimal in any $G_n$
with $n>N$. Thus, the value is uniform.
\end{proof}

\section{Leavable games}

Leavable games are cumulative games in which one of the players,
say the Maximizer, but not his opponent is allowed to leave the
game at any stage. The obvious way to model this class of games
would be to add a ``stopping'' stage between any two original
stages, where the Maximizer will choose to either ``stop'' or
``continue'' the game. However, we would also like to force the
Maximizer to ``stop'' at some stage. Unfortunately, it is
impossible to do so and still remain within the realm of
cumulative games, so we will have to deal with it a bit
differently.

Leavable games were introduced by Maitra and Sudderth (1992) as an
extension to similar concepts in the theory of gambling. They
proved that a leavable game with complete information and finite
action sets has a value. We will prove that the same is true for
leavable games with incomplete information.

Let $G$ be a cumulative game with incomplete information. A
\emph{stop rule} for player $i$ is a function
$s:\widetilde{H}^i\rightarrow \mathbb{N}$ such that if $s(h)=n$
and $h'$ coincides with $h$ in the first $n$ coordinates, then
$s(h')=n$. A \emph{leavable game with incomplete information}
$L(G)$ is given by a cumulative game with incomplete information
$G$ but is play differently, as follows. Instead of playing in
stages, both players choose their behavioral strategies
simultaneously with the Maximizer also choosing a stop rule $s$.
The game is played according to these strategies and the payoff is
$\rho(h^1)=\sum_{i=0}^{s(h^1)} f(h_n)$ where $h^1$ is the
Maximizer's private infinite history.

\begin{theorem}
A bounded leavable game with incomplete information and finite
signal sets has a value and that value is uniform. Furthermore, an
optimal strategy exists for the Minimizer.
\end{theorem}
\begin{proof}
The proof is essentially identical to the proof of Theorem 1.
$L_n$ is Defined to be the game where the Maximizer is forced to
choose a stop rule $\leq n$. $L_n$ is thus equivalent to $G_n$ in
the proof of Theorem 1.

The major point we should observe is that if $A^1$ and $S^1$ are
finite, any stop rule $s:\widetilde{H}^1\rightarrow \mathbb{N}$ is
uniformly bounded: $\exists B \forall h\in\widetilde{H}^1 \ \
s(h)<B$. This implies that any pure strategy for the Maximizer in
$L$ actually belongs to some $L_n$. Therefore, a strategy
$\sigma^2$ for the Minimizer with $val_{L_n}(\sigma^2)\leq v$ for
all $n$, has $val_L(\sigma^2)\leq v$.
\end{proof}

\section{Counterexamples}

The question arises whether positiveness is an essential or just a
technical requirement. Both our proof and the alternative proof
outlined need the positiveness in an essential way, but still is
it possible that every cumulative game have a value?

The answer is Negative. We shall provide a simple counterexample
of a cumulative game (actually a stopping game, see Dynkin (1969))
with incomplete information without a value.

The game is as follows: at the outset of the game a bit (0 or 1)
$b$ is chosen randomly with some probability $p>0$ to be 1 and
probability $1-p$ to be 0. the Maximizer is informed of the value
of $b$ but not the Minimizer. Then the following game is played.
At each odd stage the Maximizer may opt to ``stop'' the game and
the payoff is -1 if $b=0$ and 1 if $b=1$. At each even stage the
Minimizer may opt to ``stop'' the game and the payoff is -1 if
$b=0$ and some $A>\frac{1}{p}$ if $b=1$.

The payoff before and after someone decides to ``stop'' the game
is zero.

This is a very simple stopping game with only one ``unknown''
parameter, yet, as we now argue, it has no value.

\begin{claim}
The upper value of this game is $p$
\end{claim}
\begin{proof}
To see that $\overline{val}(G)\leq p$ let the Minimizer's strategy
be to continue at all stages. The Maximizer cannot gain more than
$p 1 + (1-p) 0=p $ against this strategy, so the upper value
cannot be higher than $p$.

On the other hand, let $\sigma$ be a strategy for the Minimizer.
It consists of $\{\sigma_i\}_{i=1}^\infty$ the probabilities of
stopping at stage $i$ and $\sigma_\infty=1-\sum_{i=1}^\infty
\sigma_i$ the probability of never choosing ``stop''.

Fix $\epsilon>0$ and let $N$ be an odd integer such that
$\sum_{i=N+1}^\infty \sigma_i < \epsilon$. Let $\tau$ be the
following strategy for the Maximizer: if $b=0$ never stop, if
$b=1$ stop at stage $N$. The payoff under $<\sigma,\tau>$ is:
\[
p\sum_{i=1}^N \sigma_i A + p (\sum_{i=N+1}^\infty \sigma_i +
\sigma_\infty) 1 + (1-p) \sum_{i=1}^\infty \sigma_i (-1) + (1-p)
\sigma_\infty 0
\]
\[= p (\sum_{i=1}^\infty \sigma_i + \sigma_\infty) + \sum_{i=1}^N
\sigma_i (pA-1) + \sum_{i=N+1}^\infty \sigma_i (p-1) \geq
p-\epsilon
\]
where the last inequality holds since $pA-1>0$ and
$\sum_{i=N+1}^\infty \sigma_i < \epsilon$.

Therefore $\overline{val}(G)\geq p$.
\end{proof}

\begin{claim}
The lower value of this game is $p-\frac{1-p}{A}$.
\end{claim}
\begin{proof}
Let the Maximizer play the following strategy: If $b=1$ stop at
time 1 with probability $1-\frac{1-p}{Ap}$ and continue otherwise.
If the Minimizer never decides to stop the payoff will be
$p(1-\frac{1-p}{Ap})1+(1-p)0=p-\frac{1-p}{A}$. If the Minimizer
decides to stop at any stage, the payoff will be
$p(1-\frac{1-p}{Ap})1 + p\frac{1-p}{Ap}A +
(1-p)(-1)=p-\frac{1-p}{A}$. Clearly any mix of these pure
strategies will also result in payoff of exactly
$p-\frac{1-p}{A}$.

To see that the Maximizer cannot guarantee more assume to the
contrary that there exist a strategy $\sigma$ for the Maximizer
with $val(\sigma)>p-\frac{1-p}{A}$. This strategy consists of the
probabilities $\{\sigma^0_i\}_{i=1}^\infty$ of stopping at stage
$i$ if $b=0$, and $\{\sigma^1_i\}_{i=1}^\infty$ if $b=1$.

By our assumption, the payoff against any strategy for the
Minimizer should be more than $p-\frac{1-p}{A}$. Let the Minimizer
always choose to continue. The expected payoff in that case is
\[
p(\sum_{i=1}^\infty \sigma^1_i)1 + (1-p)(\sum_{i=1}^\infty
\sigma^0_i)(-1)>p-\frac{1-p}{A} ,
\]
which implies
\[
\sum_{i=1}^\infty \sigma^1_i > 1 - \frac{1-p}{Ap} .
\]
Let $N$ be sufficiently large such that $\sum_{i=1}^N \sigma^1_i >
1 - \frac{1-p}{Ap}$. Consider the following strategy for the
Minimizer: continue until stage $N$ and then stop. The payoff will
be
\[
p(\sum_{i=1}^N \sigma^1_i)1 + p(1-\sum_{i=1}^N \sigma^1_i)A +
(1-p)(-1)
\]
\[
=p + p(1-\sum_{i=1}^N \sigma^1_i)(A-1) + (1-p)(-1)
\]
\[
<p+p\frac{1-p}{Ap}(A-1)+p-1=p-\frac{1-p}{A} ,
\]
a contradiction.
\end{proof}

\end{document}